\documentclass[pmlr]{jmlr} 

\newcommand{\tagdsmode}{nonproceedings}

\makeatletter
\newcommand{\tagdssubmission}{submission}
\newcommand{\tagdsproceedings}{proceedings}

\ifx\tagdsmode\tagdsproceedings

\else\ifx\tagdsmode\tagdssubmission
  \def\ps@jmlrtps{%
    \let\@mkboth\@gobbletwo
    \def\@oddhead{\scriptsize Under Review at the 2nd Conference on Topology, Algebra, and Geometry in Data Science\hfill}%
    \let\@evenhead\@oddhead
    \def\@oddfoot{}%
    \let\@evenfoot\@oddfoot
  }

\else
  \def\ps@jmlrtps{%
    \let\@mkboth\@gobbletwo
    \def\@oddhead{}%
    \let\@evenhead\@oddhead
    \def\@oddfoot{}%
    \let\@evenfoot\@oddfoot
  }
  \renewcommand*{\@jmlrenddoc}{}
\fi\fi
\makeatother

\usepackage{longtable}

\usepackage{booktabs}
\usepackage{siunitx}

\theorembodyfont{\upshape}
\theoremheaderfont{\scshape}
\theorempostheader{:}
\theoremsep{\newline}

\jmlrvolume{334}
\jmlryear{2026}
\jmlrworkshop{Topology, Algebra, and Geometry in Data Science}

\title[Dimensionless Controls of Plasticity]{Dimensionless Controls of Plasticity Under Alternating Tasks: From Evolutionary Biology to Continual Learning}

\ifx\tagdsmode\tagdssubmission

\else

  \author{\Name{Owen Skriloff} \Email{oskriloff@uchicago.edu} \\
   \addr University of Chicago}

\fi

\begin{document}

\maketitle

\begin{abstract}
Plasticity under changing environments is central to both evolutionary biology and continual learning. Motivated by recent work on genotype--phenotype maps, we study a minimal deep-learning analogue where a network is trained alternately on two Boolean label sets, and ask which biological controls of plasticity survive the translation to gradient descent. Reinterpreting four proposed biological factors as quantities of training dynamics, we find the system reduces to two dimensionless controls: the \emph{task disagreement} $r$, the fraction of disagreeing labels, and the \emph{reach} $\eta T$, the product of learning rate and switching period. We derive two bounds on plasticity: $r$ alone fixes an extremal geometric floor on the utopia distance, while $r$ and $\eta T$ jointly bound forgetting. Across 9{,}720 trajectories, an ANOVA confirms that $r$, $\eta$, and $T$ dominate, while the effect of neutral-set size (emphasized in the biological setting) is negligible. The optimal reach itself follows an approximate inverse power law $\eta T^{*}\propto r^{-1.18}$, yielding a heuristic that sets the optimal reach $\eta T^*$ from the task disagreement alone. The analogy that survives is therefore dynamical rather than geometric, and our setting enables a view of plasticity through the lens of other driven systems in physics and engineering.
\end{abstract}

\section{Introduction}

A longstanding question in the study of learning is how far universality extends across systems and substrates~\citep{learningmech26, repsalign24, platonic24, diffconsistency24, spinglass23}. Recent work argues that the empirical successes of deep learning will give rise to a scientific \emph{mechanics of learning}~\citep{learningmech26}: just as physical mechanics is organized around canonical solvable models (the harmonic oscillator, the hydrogen atom, the Ising model) and a few dimensionless quantities (the Reynolds number), the mechanics of learning aims for a compact catalog of laws governing training. This work contributes one entry to that catalog, a toy model for plasticity in deep neural networks.

We study \emph{plasticity}, a phenomenon investigated in parallel in continual learning and evolutionary biology, with mounting evidence suggesting shared structure across the two fields. We build on~\citet{PlasticGP}, who show that biological plasticity emerges when environments alternate and identify four controlling factors (neutral set size (NSS), Hamming distance, switching period, and mutation rate). We translate their setting to a deep network trained alternately on two label sets and ask which of these controls survive the passage to gradient descent.

In continuous time, the system is a switched first-order gradient flow~\citep{latz2021}, structurally identical to the periodically forced overdamped systems studied in statistical physics~\citep{joubaud2015, jung1993}, electrical engineering~\citep{oppenheim1997}, and chemical engineering~\citep{levenspiel1999}. We exploit this analogy to identify the dimensionless quantities advocated by~\citet{learningmech26}. By reading the learning rate as an inverse resistance and the switching period as a relaxation time, their product, a dimensionless reach in parameter space, represents a Damk\"ohler number~\citep{levenspiel1999} that sets how far the network relaxes within one phase. Our contributions are as follows:
\begin{itemize}
    \item \textbf{Two dimensionless controls.} We reduce the four biological factors of~\citet{PlasticGP} to two dimensionless controls of alternating training (task disagreement $r$ and reach $\eta T$), and show the remaining factors are subdominant. The optimal reach further obeys an approximate power law $\eta T^{*}\propto r^{-1.18}\approx 1/r$, so the optimal reach $\eta T^*$ can be set from $r$ alone.
    \item \textbf{Bounds on plasticity.} We prove that $r$ alone fixes a tight geometric floor on the utopia distance, $d_U \geq \sqrt{2}(1-2^{-r})$ (Theorem~\ref{thm:utop}), while $r$ and $\eta T$ jointly bound forgetting, $\Delta a_1 \lesssim e^{r\,\eta T\,M}-1$ (Theorem~\ref{thm:rob}).
    \item \textbf{Empirical validation and a divergence from biology.} Across 9{,}720 trajectories, an ANOVA shows that the \emph{static} landscape quantity of neutral set size, critical in biology, has negligible effect, while the \emph{dynamic} quantities of task disagreement, learning rate, and switching period dominate.
    \item \textbf{Toward a mechanics of learning.} Together, these results (disentangled hyperparameters, dimensionless controls, and scaling laws) provide a new solvable setting for plasticity that mirrors driven systems in physics and engineering, supporting the emerging mechanics of learning.
\end{itemize}

\section{Related Work}
\label{sec:related}

Neural networks arose from a combination of biomimicry, through artificial neuron models, and statistical physics, through spin-glass models~\citep{hopfield82}, spawning rich literatures linking machine learning to physics~\citep{statmech20} and neuroscience~\citep{neuro17}. Evolutionary biology and learning have likewise been connected through genetic algorithms~\citep{gareview21}, shared simplicity bias~\citep{Symmetry}, and readings of the central dogma as a generative model~\citep{geneticcode25}.

A strong connection has also formed between plasticity in evolutionary adaptation and in continual learning. The loss of plasticity, or a decline in a network's ability to learn new tasks, is well documented~\citep{plasticlost24, understandplastic23}. The field has largely targeted catastrophic forgetting under task change~\citep{forgetting17, surgury20, plasticlost24}, with recent work tying forgetting to the geometry of loss landscapes~\citep{connectforget21} and cyclical environmental change~\citep{abbas2023}.

These fields are further linked by geometric similarities between fitness landscapes (with genes as parameters) and loss landscapes of neural networks (with weights and biases as parameters). The equivalence of local minima~\citep{surflocalglobal14, localglobalmins17}, connectedness of low-loss regions~\citep{modeConnect2018, convexModeConn22, mechanisticModeConn23, lmc23}, and link between flatness and generalization~\citep{Hochreiter1997Flat, pmlr-v272-haddouche25a, pouplin2023curvaturelosslandscape} all recall the large, connected \emph{neutral sets} of genes that map to the same optimal phenotype in genotype--phenotype (GP) maps~\citep{ReviewGPmaps}. \citet{PlasticGP} show that plasticity emerges when environments alternate, and identify neutral set size as one of four controlling factors, along with Hamming distance, switching period, and mutation rate. We follow the precedent of taking inspiration from biology and test whether plasticity emerges in an analogous continual learning setting, studying how controls survive in an alternative substrate.

\section{Notation and Key Quantities}
\label{sec:notation}

\paragraph{Plasticity.} Plasticity has two related but distinct meanings. In continual learning it denotes a network's sustained ability to learn, whose decline is catastrophic forgetting. In the evolutionary setting of~\citet{PlasticGP} it denotes a single genotype performing well across multiple environments at once, or adapting quickly to changing conditions.

\paragraph{Accuracy proxies.} We train a network alternately on two label sets $y_1$ and $y_2$ over $N$ inputs, switching every $T$ steps. Writing $\ell_j(t)$ for the loss on task $j$ at step $t$, we form accuracy proxies $a_j(t)=e^{-\ell_j(t)}\in(0,1]$, so $a_j=1$ is perfect performance on task $j$. These proxies are bounded, monotone transformations of the loss, confine the trajectory to the unit square where the utopia geometry is defined, and echo the Boltzmann-probability reading of performance in~\citet{PlasticGP}. The main conclusions are not specific to this choice: other monotone maps to the unit square (a reciprocal transform, direct $0/1$ accuracy, and a Brier-style score) yield the same qualitative trends and rank-correlate strongly with $a_j$ (Appendix~\ref{app:accuracy}). The pair $(a_1(t),a_2(t))$ traces a trajectory in the unit square, and joint mastery corresponds to the \emph{utopia} point $(1,1)$ (Figure~\ref{fig:trajectory}).

\begin{figure}[htbp]
    \centering
    \includegraphics[width=0.83\textwidth]{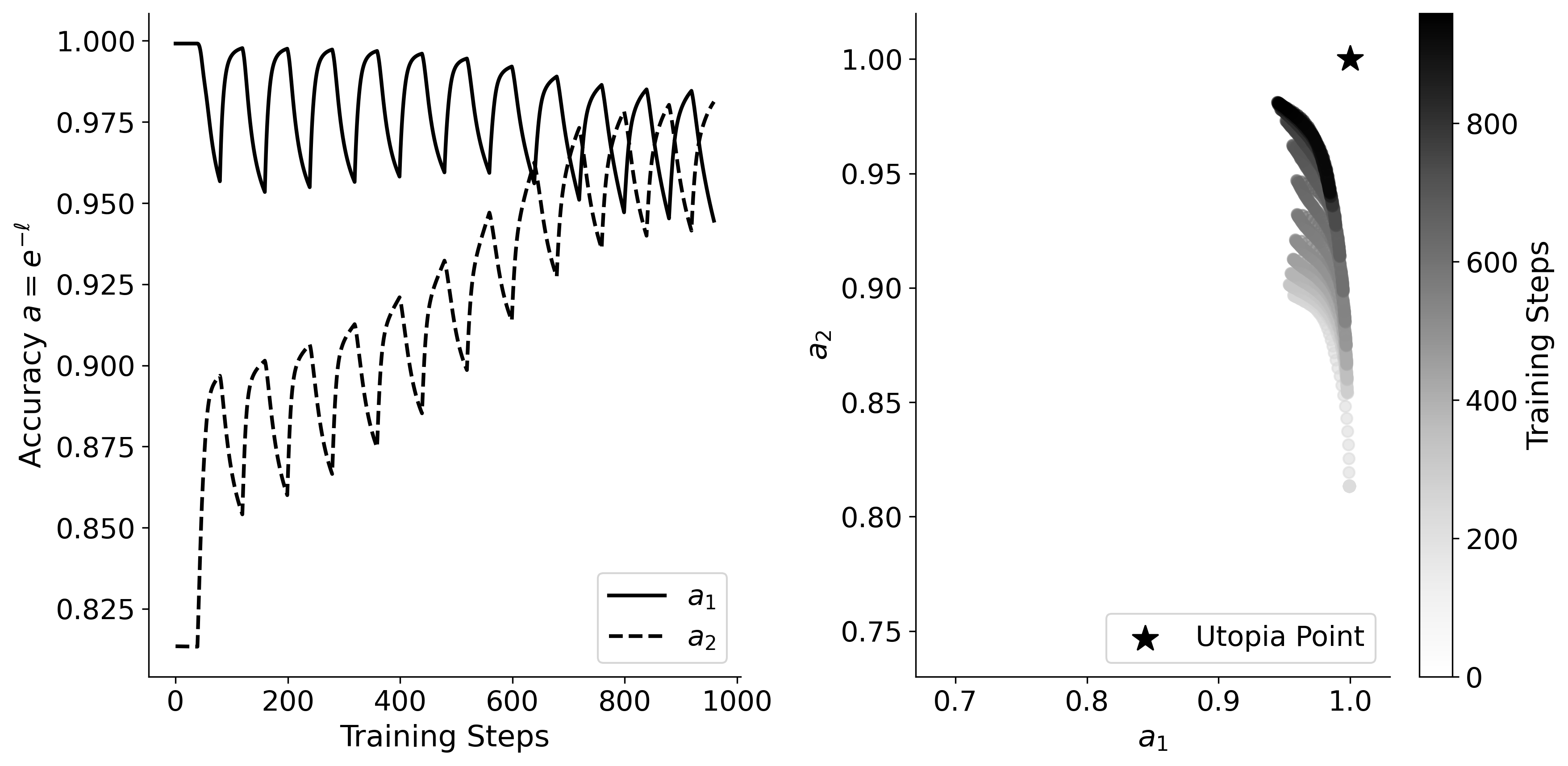}
    \caption{Example trajectory of alternating training on two functions. The left panel shows the accuracy proxies $a_1(t)$ and $a_2(t)$ over iterations. The right panel shows the corresponding phase plane trajectory and utopia point over iterations.}
    \label{fig:trajectory}
\end{figure}

\paragraph{Utopia distance.} The \emph{instantaneous utopia distance} $\delta_U(t) = \left\|(a_1(t),a_2(t))-(1,1)\right\|_2$ is the Euclidean distance from the current state to the utopia point. The \emph{utopia distance} $d_U$ captures how close the network comes to mastering both tasks at once: it is the time average of $\delta_U(t)$ at steady state, over the final two switching periods (ending at step $T_*$),
\[
    d_U = \frac{1}{2T}\sum_{t=T_*-2T}^{T_*-1}\delta_U(t),
\]
so $d_U$ is small only when both $a_1$ and $a_2$ stay near $1$ throughout the cycle, penalizing a trajectory that excels at one task at the expense of the other. A lower $d_U$ indicates higher plasticity.

\paragraph{Forgetting amplitude.} The \emph{forgetting amplitude} $\Delta a_1$ captures how much task-1 performance is lost when training switches to task 2: the peak-to-peak swing in $a_1$,
\[
    \Delta a_1 = \sup_{t,s\in[0,2T]}|a_1(t)-a_1(s)|,
\]
analogous to catastrophic forgetting~\citep{forgetting17,surgury20} and to the robustness measure of~\citet{PlasticGP}.

\paragraph{Dimensionless controls.} Two dimensionless quantities govern these outcomes. The \emph{task disagreement} $r=d_H/N\in[0,1]$ is the fraction of inputs on which $y_1$ and $y_2$ differ, with $d_H$ their Hamming distance (see~\citealp{hambook16}). The \emph{reach} $\eta T$ multiplies the learning rate $\eta$ ($\mathrm{step}^{-1}$) by the switching period $T$ (steps), measuring how far the network travels in parameter space during a single phase.

\paragraph{Analysis of variance.} To test which factors drive the empirical outcomes (Section~\ref{sec:results}), we use a Type II analysis of variance (ANOVA), reporting for each factor the effect size $\omega^2$, the fraction of outcome variance it explains. A large $\omega^2$ marks a dominant factor, letting us compare the influence of $r$, $\eta$, $T$, neutral set size, and complexity.

\section{Experimental Setup}

We approximate 7-bit Boolean functions ($N=2^7=128$) with a neural network (architecture and training details in Appendix~\ref{app:A}, Table~\ref{tab:hparams}). For each function $f$, the neutral set size $\mathbb{P}(f)=\int_\Theta \mathbb{I}\{\ell_f(\theta)=0\}\,p(\theta)\,d\theta$ is estimated by Monte Carlo sampling from a Gaussian prior $p(\theta)$~\citep{OcRazNN25, flatnessbad21}. Since simplicity bias couples Lempel--Ziv complexity $\tilde{K}(f)$ tightly to neutral set size~\citep{BoundP2018, OcRazNN25}, we decouple them by sampling function pairs within three complexity bins at varying probabilities. As all pairs have nonzero Hamming distance and cannot have overlapping neutral sets, a pair's total neutral set size is $\nu = \mathbb{P}(f_1) + \mathbb{P}(f_2)$.

We select 27 function pairs spanning low, medium, and high values of neutral set size, complexity, and Hamming distance $d_H$ in a $3\times3\times3$ grid (Appendix~\ref{app:A}, Table~\ref{tab:pairs}). After a pretraining phase on the first function, each experiment alternates SGD between the two label sets every $T$ steps until reaching steady state. Sweeping six learning rates and six switching periods (Table~\ref{tab:hparams}) with ten Monte Carlo trials per $(\eta, T)$ pair gives 9{,}720 trajectories, for each of which we compute $d_U$ and $\Delta a_1$.

\section{Theory}

All proofs are in Appendix~\ref{app:proofs}. We first upper bound the forgetting amplitude $\Delta a_1$, giving its scaling behavior:

\begin{theorem}
\label{thm:rob}
Suppose the trajectory has reached a period-$2T$ steady-state orbit, so that $a_1(0)=a_1(2T)$; that $a_1(t)$ monotonically decreases during the $y_2$ training phase; and that the loss is $L$-smooth and the network $f$ is $C$-smooth on the orbit. Let
\[
M=\sup_{\theta \in \mathrm{Orbit}}\| \nabla \ell_{1}(\theta)\|_{2} \max_{i\in\perp} \|J_i(\theta)\|_2 \le LC < \infty,
\]
where $\perp$ denotes the set of indices where the labels $y_1$ and $y_2$ differ, and $J_i(\theta)$ is the Jacobian of the network output with respect to parameters $\theta$ at index $i$. Here, $M$ is the maximum amplitude of the forcing term in the linearized response of $a_1$ to training on $y_2$.

Then $\Delta a_1\leq c_1$, where $c_1=a_1(0)(e^{r\eta T M + O(\eta^2 T)}-1)$. Further, for sufficiently small $\eta$, we have $c_1\approx a_1(0)(e^{r\eta T M}-1)$.
\end{theorem}

Two caveats temper the bound. First, the monotonic decrease of $a_1(t)$ during the $y_2$ phase is not guaranteed, though it holds in $97.5\%$ of phases within a test grid of 56{,}748 steady-state $y_2$ phases (Appendix~\ref{app:monotonicity}). The assumption is thus a close but not universal description of the dynamics. Second, $M\le LC$ is loose, so we use Theorem~\ref{thm:rob} for its scaling rather than as a quantitative bound, verified in Figure \ref{fig:bounds} (left) with the uniform proxy $M=100$.

Next, we derive a tight lower bound on $d_U$ depending only on $r$ (Figure \ref{fig:bounds}, right):

\begin{theorem}
\label{thm:utop}
If $0<r<1$, then $d_U\geq \sqrt{2}(1-2^{-r}).$
\end{theorem}

\begin{figure}[htbp]
    \centering
    \includegraphics[width=0.83\textwidth]{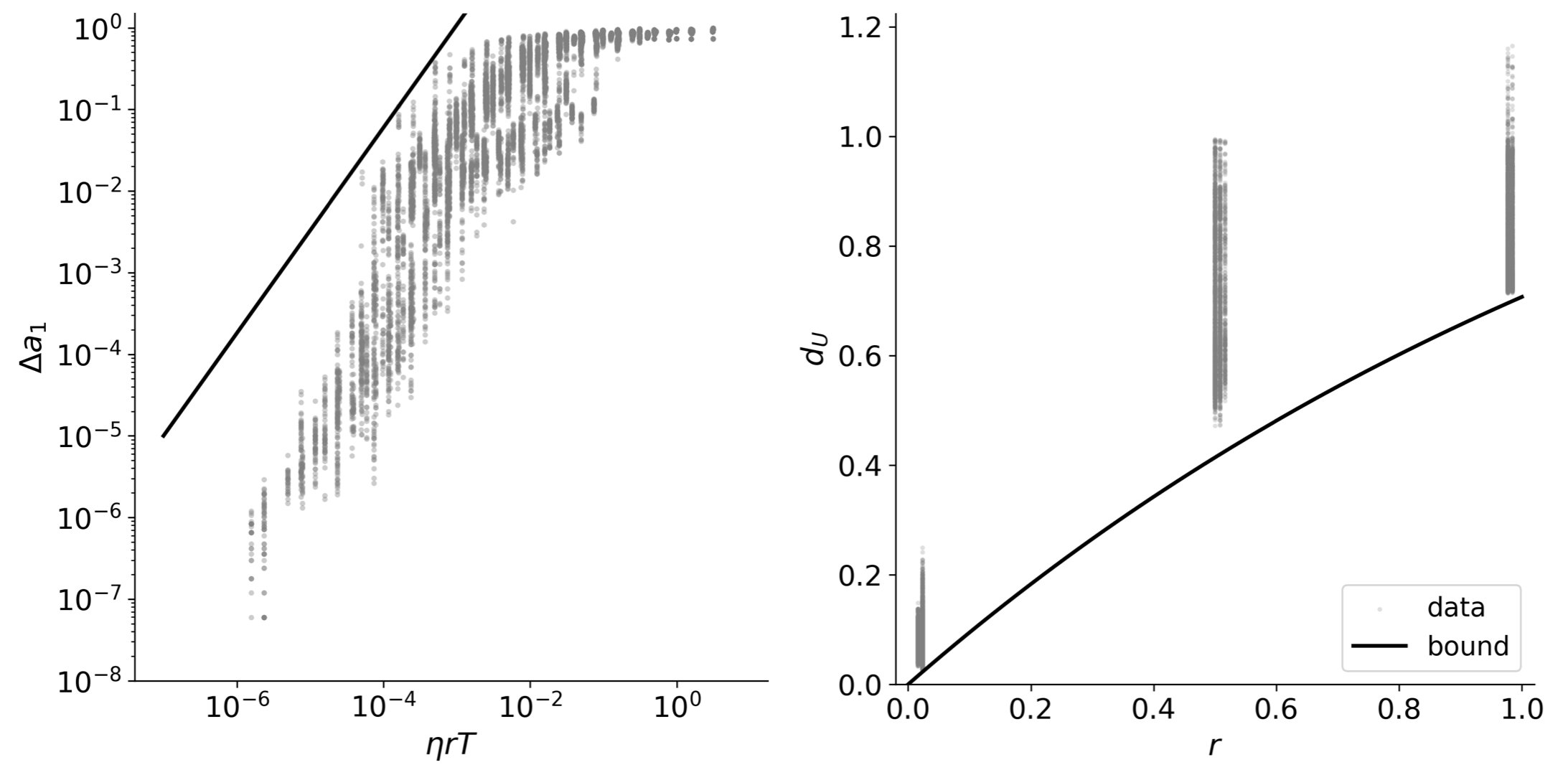}
    \caption{Empirical demonstration of the bounds of Theorems \ref{thm:rob} (left) and \ref{thm:utop} (right) across the 9{,}720 trajectories. For Theorem~\ref{thm:rob}, we use the uniform estimate $M=100$.}
    \label{fig:bounds}
\end{figure}

\section{Results}
\label{sec:results}

To test whether the theory's exclusive dependence on $\eta$, $r$, and $T$ is an artifact of the proof techniques or a true phenomenon, we ran a Type II ANOVA (Table \ref{tab:ANOVA}). It confirms both predicted dependencies: $\eta$ and $T$ have much larger effect sizes on $\Delta a_1$ than on $d_U$, while $r$ dominates $d_U$ far more than $\Delta a_1$. The bounds of Theorems \ref{thm:rob} and \ref{thm:utop} thus capture the dominant dependence of each outcome on the controls.

\begin{table}[h]
\caption{Type II ANOVA effect sizes for the two log-transformed plasticity outcomes. Models use complexity bin, neutral-set-size class, Hamming-distance class, learning rate, and switching period as categorical predictors. Reported values are $\omega^2$ effect sizes. Significance levels are denoted by $^{***}$ for $p<0.001$. Note that while $\tilde K$ and NSS yield statistically significant effects, their effect sizes are negligible compared to the other factors.}

\label{tab:ANOVA}
\centering
\begin{tabular}{lcc}
\toprule
\textbf{Factor} & $\log \Delta a_1$ & $\log d_U$ \\
\midrule
$\tilde{K}$ bin & $1{\times}10^{-4}$ & $3{\times}10^{-4\,***}$ \\
NSS ($\nu$) class & $3{\times}10^{-4\,***}$ & $1{\times}10^{-4\,***}$ \\
Hamming ($r$) class & $0.4966^{***}$ & $0.9422^{***}$ \\
$\eta$ & $0.2178^{***}$ & $0.0019^{***}$ \\
$T$ & $0.1348^{***}$ & $0.0030^{***}$ \\
Residual & $0.1503$ & $0.0525$ \\
\midrule
$R^2$ & $0.8497$ & $0.9475$ \\
\bottomrule
\end{tabular}
\end{table}

Notably, the ANOVA rejects the hypothesis that neutral set size $\nu$ and complexity $\tilde{K}$ strongly influence plasticity as they do in biology: despite statistical significance from the large sample, their effect sizes are negligible compared to the other factors. This points to a structural difference between evolution on the fitness landscape of genotype-phenotype maps and learning on the loss landscape of neural networks.

Despite the dominant effect size of $r$, the spread of $d_U$ above the theoretical floor within each $r$ bin (Figure \ref{fig:bounds}) shows that $r$ alone does not fully determine $d_U$. Where a trajectory actually sits above the extremal floor given by Theorem~\ref{thm:utop} is governed by the reach $\eta T$, as we show next. A heatmap of $d_U$ over the $(T,\eta)$ space (Figure \ref{fig:heatmap}, left) reveals a diagonal of optimal (lower) values along which $\eta T$ is approximately constant, suggesting that the product $\eta T$ governs optimal plasticity.

Plotting $d_U$ against $\eta T$ across the $r$ spectrum (Figure \ref{fig:heatmap}, center) confirms this. As $r$ increases, the minimum attainable utopia distance $d_U$ \emph{rises}, tracking the Theorem~\ref{thm:utop} floor, while the optimal $\eta T$ (the value minimizing $d_U$) \emph{falls}, and the optimum sharpens as the $r$-level curves grow more convex. Across 37 task pairs spanning 17 disagreement levels, this optimal reach follows an approximate power law (Figure \ref{fig:heatmap}, right), $\eta T^{*}\propto r^{-1.18}$ (pair-bootstrap 95\% CI $[-1.34,-1.01]$, $R^2=0.86$). Together, theory and experiment show a clean division of labor: task disagreement alone sets the geometric floor of $d_U$, the reach governs where the trajectory sits above it, and their product $r\eta T$ governs the scaling of the forgetting amplitude $\Delta a_1$.

\begin{figure}[ht]
    \centering
    \includegraphics[width=\textwidth]{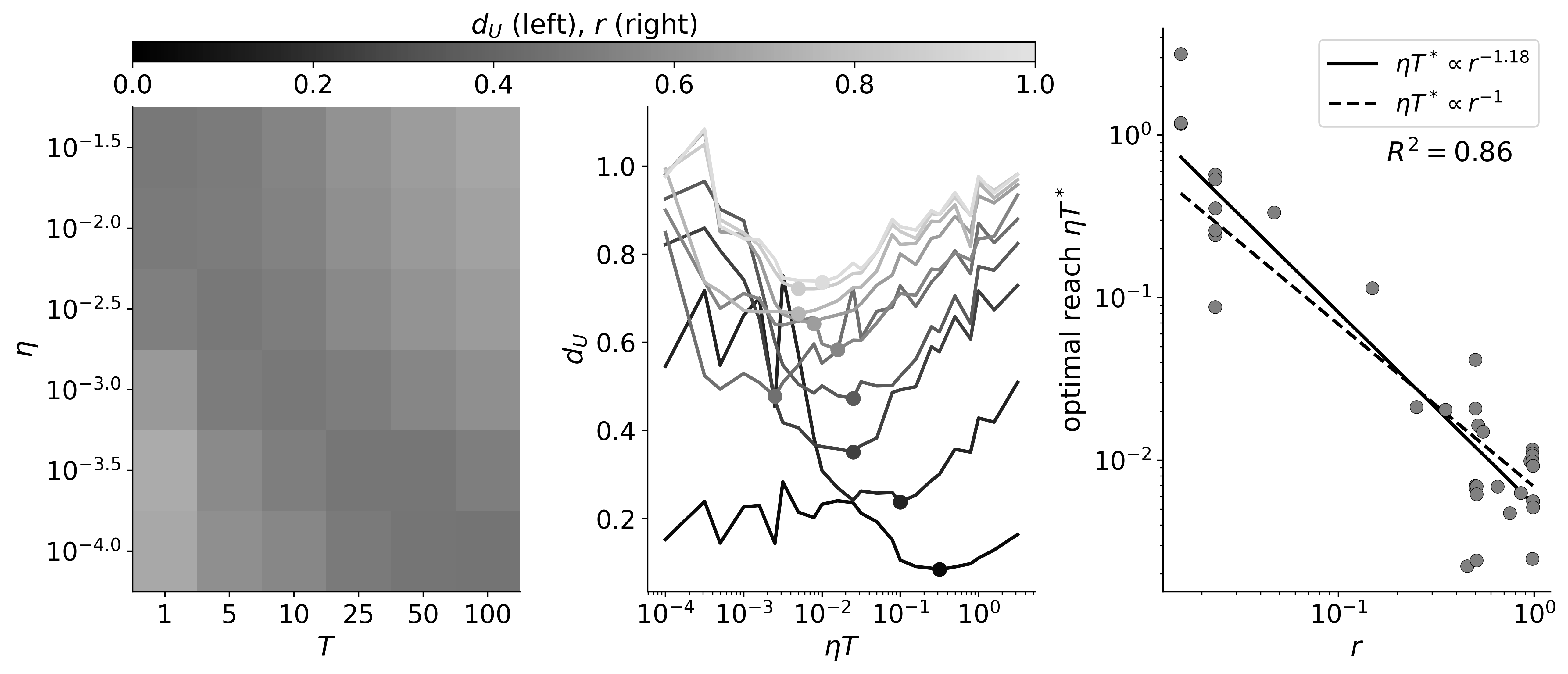}
    \caption{Heatmap of $d_U$ values across $(T,\eta)$ space with the most plastic trajectories (lowest $d_U$ values) occurring along a diagonal (constant $\eta T$) (left). Task disagreement $r$ level curves of the utopia distance $d_U$ versus nominal reach $\eta T$ (center) show that as $r$ grows the optimum rises and shifts to smaller $\eta T$ while curves become more convex. The optimal reach $\eta T^{*}$ versus $r$ over 37 task pairs (right) shows a power-law fit $\eta T^{*}\propto r^{-1.18}$ (95\% CI $[-1.34,-1.01]$, $R^2=0.86$).}
    \label{fig:heatmap}
\end{figure}

Beyond this primary grid, a set of smaller robustness studies (Appendix~\ref{app:robust}) indicates the two controls are not artifacts of the specific setup. The reach-governed picture persists under minibatch SGD, momentum, and Adam, and under held-out continuous-input tasks where performance is measured on unseen test data (mean train--test gap in $d_U$ below $0.01$). Varying network width and depth over a $16\times$ range in parameter count leaves $d_U$ essentially unchanged, giving direct empirical support for network-size independence. These small studies are reported as consistency checks rather than conclusive evidence of generality.

\section{Discussion}

\paragraph{Plasticity across substrates.} We translated the plasticity model of~\citet{PlasticGP} from evolutionary biology to deep learning, enabling a direct comparison across substrates. While Hamming distance, mutation (learning) rate, and switching period control plasticity in both settings, neutral set size plays a negligible role in the learning setting, suggesting a geometric divergence between the parameter spaces of genes and network weights despite the similarity of their fitness and loss landscapes.

The negligible role of neutral set size is the most surprising departure from biology. One plausible mechanism is that population dynamics deploy many trajectories that sample the neutral set volume, whereas a single SGD run occupies one point in it, governed by local forcing over global geometry. This predicts a larger role for neutral set size in ensemble methods that mimic population-based training as in \citet{ensembles20}, or under large explicit noise injection. We leave the testing of this prediction to future work.

\paragraph{Reach as a physical control parameter.} Prior work by~\citet{latz2021} casts the continuous-time limit of SGD as phenotypic evolution under randomly fluctuating environments. We replace the random fluctuation with a deterministic square wave, a canonical forcing signal in physics and engineering. The linearized response to such forcing scales like $\tanh(x)$~\citep{amplituderesp25}, which Theorem~\ref{thm:rob} bounds via $\tanh(x) \leq e^x - 1$, suggesting room to tighten in the linear regime.

Identifying the reach $\eta T$ as a control parameter draws a further analogy. The optimal learning rate $\eta$ is closely tied to the inverse local curvature (Lipschitz constant) of the loss landscape~\citep{relaxlr26, edgestability21}. Reading curvature as the system's ``resistance'' makes $\eta$ an inverse resistance, speeding return to minima after switching. $T$ then becomes the relaxation time, with their product $\eta T$ being the distance traveled in one phase, the reach. In chemical engineering this is exactly the first Damköhler number~\citep{levenspiel1999}, the ratio of residence to reaction time. The optimal reach spans roughly $0.005$--$0.02$ at high $r$ and is of order $0.1$--$1$ at low $r$, where the optimum is only weakly identified because the reach curve flattens. Thus, $\eta T \lesssim 1$ throughout and $\eta T \ll 1$ wherever the optimum is sharp, identifying the system as reaction-limited: for optimal plasticity, the phase must be short enough that relaxation cannot complete to a local minimum.

The convexity of the reach curve follows from the same picture: at low $r$ the system relaxes to a shared minimum and the curve is nearly flat, while at higher $r$ too little reach underfits both tasks and too much reach overfits the current task and forgets the other, both inflating $d_U$. The optimal reach falls as disagreement grows since the system becomes more sensitive to overfitting due to the greater distance between the two tasks' minima.

Further, the optimal reach falls with $r$ as an approximate power law $\eta T^{*}\propto r^{-1.18}$, close to the inverse law $\eta T^{*}\propto 1/r$ and suggesting a universal scaling for alternating tasks. This yields a design heuristic: given the many methods for estimating $\eta$ from curvature~\citep{relaxlr26, edgestability21}, the optimal switching period follows from the task disagreement alone, without probing the loss landscape or architecture. Combined with the two theorems, it collapses both plasticity outcomes to functions of $r$: the utopia floor $d_U \geq \sqrt{2}(1-2^{-r})$ and the forgetting amplitude $\Delta a_1 \lesssim e^{r\eta T^{*}M}-1 \approx e^{c M}-1$ for some constant $c$.

\paragraph{Connections to continual learning.} Our results also quantify continual-learning findings. Forgetting as curvature-weighted parameter displacement underlies the loss-landscape analysis of \citet{connectforget21}, who show empirically that batch size, learning rate, and flatness affect forgetting. Theorem \ref{thm:rob} gives a closed-form scaling in reach $\eta T$ and disagreement $r$. Elastic Weight Consolidation (EWC; \citealp{forgetting17}) adds a remembering term penalizing changes to parameters important for the previous task, and Theorem~\ref{thm:rob} adapts directly to quantify the gain. Under a diagonal-Fisher approximation near an optimum, an EWC penalty of strength $\lambda$ replaces the binary cross entropy (BCE) bound's exponent $\eta T(rM - E)$ with $\eta T(rM - E(1+\lambda))$, where $E$ is the mean squared task-1 gradient norm over the phase, strictly lowering the forgetting bound (Corollary~\ref{cor:EWC}, Appendix~\ref{app:proofs}).

Continual Learning with Experience and Replay (CLEAR; \citealp{reaplay18}) interweaves data from both tasks, much like our task switching. Since the empirically optimal reach under $d_U$ is nonzero, effective plasticity on disagreeing tasks requires a finite amount of learning within each alternating phase, giving quantitative support for CLEAR. Our results further suggest that optimal hyperparameters for CLEAR, or for alternating environments used directly as a continual-learning method, could be estimated from theory alone using our hyperparameter heuristic when the tasks disagree.

\paragraph{Toward a mechanics of learning.} Taken together, these results add direct evidence to four of the five strands of the mechanics-of-learning program~\citep{learningmech26}: a solvable, idealized setting for plasticity, simple governing laws (Theorems \ref{thm:rob} and \ref{thm:utop}), a reduction of hyperparameters to two dimensionless controls, and a link between evolutionary biology, chemistry, and learning hinting at substrate universality.

\section{Conclusion and Future Work}

We translated a model of biological plasticity to deep networks trained on alternating Boolean tasks and found that plasticity under task switching is governed by two dimensionless controls: the task disagreement $r$ and the reach $\eta T$. We proved that $r$ alone fixes a tight geometric floor on the utopia distance, while $r$ and $\eta T$ jointly bound forgetting. An ANOVA confirmed that these controls dominate the empirical behavior while the static landscape quantities central to the biological setting do not. The analogy that survives the translation is therefore dynamical rather than geometric, placing single-trajectory learning alongside driven systems in physics and engineering.

Since any computational data can in principle be written as a Boolean function, the setting is a candidate minimal model. Finer sweeps of the reach's interaction with $r$ could sharpen the exponent, but we conjecture more strongly that setting the reach to the disagreement-matched value $\eta T \approx 1/r$ provides near-optimal conditions for continual learning on alternating tasks (phases short enough to avoid overfitting the current task, yet long enough to keep learning). Thus, hyperparameter constraints follow from $r$ alone, echoing the transfer of maximal-update parameterization ($\mu$P) scaling heuristics~\citep{muP2022}.

The disagreement $r=d_H/N$ itself invites generalization. Because $d_H$ is a metric on binary label vectors, the framework may be generalized using other dissimilarity measures. For BCE, information-theoretic divergences or distances are natural candidates. The discrete bound $a_1 a_2 \le 4^{-r}$ (Lemma~\ref{lem:constraint}, Appendix~\ref{app:proofs}) suggests one route toward extending the framework to soft labels and continuous outputs. We present one alternative continuous metric in Appendix~\ref{app:robust} where $r$ is replaced by the angle between dividing hyperplanes of two linear classifiers.

The switching period $T$ should be governed by the non-commutativity of the two gradient flows. The Lie bracket $[\nabla\ell_1,\nabla\ell_2]$~\citep{commute2025}, accumulated over one period, links the reach to $\Delta a_1$. Since this quantity vanishes when the flows commute, we recover the $r\to0$ limit of Theorem~\ref{thm:rob}, but further analysis could explicitly link the reach, forgetting amplitude, and the Lie bracket for general tasks. This would extend the framework to general loss functions and continuous outputs, and could provide a geometric interpretation of the reach as a curvature-weighted distance in parameter space. 

The framework may also extend to generalization error. Since forgetting is a performance loss from distribution shift between tasks \citep{doan2021forgetting}, $\Delta a_1$ and $d_U$ may measure robustness to distribution shift and noisy labels, underscored by the out-of-distribution forgetting continual learners show under small intra-class shifts \citep{guo2023oodf}.

\acks{Thank you to Abiy Tasissa for helpful discussions and feedback on the manuscript.}

\bibliography{main}

\appendix

\section{Computational Experiment Details}
\label{app:A}

\begin{table}[htbp]
\centering
\caption{Architecture and training hyperparameters.}
\label{tab:hparams}
\begin{tabular}{ll}
\toprule
\textbf{Setting} & \textbf{Value} \\
\midrule
Hidden layers & 10 \\
Layer width & 40 \\
Activation & Tanh \\
Initialization & i.i.d.\ $\mathcal{N}(0, 1/\text{width})$ \\
Loss function & Binary cross-entropy (BCE) \\
Optimizer & SGD (full batch) \\
Learning rates $\eta$ & $\{10^{-4},\,10^{-3.5},\,10^{-3},\,10^{-2.5},\,10^{-2},\,10^{-1.5}\}$ \\
Switching periods $T$ & $\{1,\,5,\,10,\,25,\,50,\,100\}$ \\
Max iterations & 2500 \\
Trials per $(\eta,T)$ & 10 \\
\bottomrule
\end{tabular}
\end{table}

\begin{table}[htbp]
  \centering
  \small
  \caption{Attributes of the 27 label pairs $(y_1, y_2)$ used in the computational experiments.}
  \label{tab:pairs}
  \begin{tabular}{llllll}
    \toprule
    \# & $\tilde{K}$ & $\mathbb{P}(y_1)$ & $\mathbb{P}(y_2)$ & $\nu$ & $d_H$ \\
    \midrule
     1 & 56.0 & $1.80\times10^{-5}$ & $6.00\times10^{-6}$ & $2.40\times10^{-5}$ &   3 \\
     2 & 56.0 & $1.00\times10^{-5}$ & $1.20\times10^{-5}$ & $2.20\times10^{-5}$ &  64 \\
     3 & 56.0 & $1.80\times10^{-5}$ & $6.00\times10^{-6}$ & $2.40\times10^{-5}$ & 125 \\
     4 & 56.0 & $2.20\times10^{-5}$ & $4.60\times10^{-5}$ & $6.80\times10^{-5}$ &   3 \\
     5 & 56.0 & $3.00\times10^{-5}$ & $3.40\times10^{-5}$ & $6.40\times10^{-5}$ &  65 \\
     6 & 56.0 & $4.60\times10^{-5}$ & $2.20\times10^{-5}$ & $6.80\times10^{-5}$ & 125 \\
     7 & 56.0 & $9.80\times10^{-5}$ & $2.36\times10^{-4}$ & $3.34\times10^{-4}$ &   2 \\
     8 & 56.0 & $5.80\times10^{-5}$ & $3.94\times10^{-4}$ & $4.52\times10^{-4}$ &  65 \\
     9 & 56.0 & $9.80\times10^{-5}$ & $2.36\times10^{-4}$ & $3.34\times10^{-4}$ & 126 \\
    \midrule
    10 & 59.5 & $6.00\times10^{-6}$ & $1.20\times10^{-5}$ & $1.80\times10^{-5}$ &   3 \\
    11 & 59.5 & $1.20\times10^{-5}$ & $6.00\times10^{-6}$ & $1.80\times10^{-5}$ &  64 \\
    12 & 59.5 & $1.20\times10^{-5}$ & $6.00\times10^{-6}$ & $1.80\times10^{-5}$ & 126 \\
    13 & 59.5 & $3.00\times10^{-5}$ & $2.00\times10^{-5}$ & $5.00\times10^{-5}$ &   2 \\
    14 & 59.5 & $5.00\times10^{-5}$ & $6.00\times10^{-6}$ & $5.60\times10^{-5}$ &  64 \\
    15 & 59.5 & $3.00\times10^{-5}$ & $2.00\times10^{-5}$ & $5.00\times10^{-5}$ & 125 \\
    16 & 59.5 & $6.00\times10^{-5}$ & $3.80\times10^{-4}$ & $4.40\times10^{-4}$ &   2 \\
    17 & 59.5 & $4.56\times10^{-4}$ & $6.00\times10^{-6}$ & $4.62\times10^{-4}$ &  66 \\
    18 & 59.5 & $4.56\times10^{-4}$ & $5.00\times10^{-5}$ & $5.06\times10^{-4}$ & 126 \\
    \midrule
    19 & 66.5 & $6.00\times10^{-6}$ & $8.00\times10^{-6}$ & $1.40\times10^{-5}$ &   3 \\
    20 & 66.5 & $8.00\times10^{-6}$ & $6.00\times10^{-6}$ & $1.40\times10^{-5}$ &  64 \\
    21 & 66.5 & $6.00\times10^{-6}$ & $8.00\times10^{-6}$ & $1.40\times10^{-5}$ & 125 \\
    22 & 66.5 & $1.20\times10^{-5}$ & $8.00\times10^{-6}$ & $2.00\times10^{-5}$ &   3 \\
    23 & 66.5 & $1.20\times10^{-5}$ & $6.00\times10^{-6}$ & $1.80\times10^{-5}$ &  64 \\
    24 & 66.5 & $1.20\times10^{-5}$ & $8.00\times10^{-6}$ & $2.00\times10^{-5}$ & 125 \\
    25 & 66.5 & $1.20\times10^{-5}$ & $1.30\times10^{-4}$ & $1.42\times10^{-4}$ &   3 \\
    26 & 66.5 & $1.30\times10^{-4}$ & $6.00\times10^{-6}$ & $1.36\times10^{-4}$ &  65 \\
    27 & 66.5 & $1.30\times10^{-4}$ & $1.20\times10^{-5}$ & $1.42\times10^{-4}$ & 125 \\
    \bottomrule
  \end{tabular}
\end{table}

\section{Proofs for Theoretical Results}
\label{app:proofs}

\subsection{Proof of Theorem~\ref{thm:rob}}

\begin{proof}
From $t=0$ to $t=T$, training proceeds by full-batch gradient descent on the $y_1$ labels. Since $\ell_1$ is $L$-smooth (an assumption of the theorem), the descent lemma guarantees that any step size $\eta < 2/L$ strictly decreases the loss at each iteration; hence $\ell_1(t)$ is monotonically decreasing and $a_1(t)=e^{-\ell_1(t)}$ is monotonically increasing on $[0,T]$. By the assumption of the theorem, from time $t=T$ to $t=2T$ we know $a_1(t)$ is monotonically decreasing. Further, by periodicity we know that $a_1(0)=a_1(2T)$. Thus, we can conclude that
\[
    \Delta a_1 = \max_{s,t \in [0,2T]}|a_1(t)-a_1(s)|=a_1(T)-a_1(0)
\]
Since the trajectory evolves by gradient descent with respect to the $y_1$ labels during the $[0,T]$ phase, we define the quantity $\varsigma$, the total change in loss over the phase, as
\[
\varsigma := \eta \sum_{t=0}^{T-1}\|\nabla\ell_1(t)\|_2^2 + O\left( \eta^2 \sum_{t=0}^{T-1}\|\nabla\ell_1(t)\|_2^2 \right)
\]
Alternatively, we can derive that $\ell_1(T)=\ell_1(0)-\varsigma$ and 
$\varsigma  =\ell_1(2T)-\ell_1(T)$ by periodicity. This allows us to write 
\begin{equation}
\label{eq1}
\Delta a_1=a_1(0)(e^\varsigma - 1).
\end{equation}

We proceed by upper bounding $\varsigma$. In this second phase where $a_1(t)$ is decreasing, we know that the increasing value of $\ell_1$ is governed by the gradient disagreement between labels $\langle \nabla\ell_1(t),\nabla \ell_2(t) \rangle$. Using a Taylor expansion, we obtain
\[
\varsigma=-\eta \sum_{t=T}^{2T-1}\langle \nabla\ell_1(t),\nabla \ell_2(t) \rangle + O\left( \eta^2 \sum_{t=T}^{2T-1}\|\nabla\ell_1(t)\|_2^2 \right).
\]
To bound $-\langle \nabla\ell_1(t),\nabla \ell_2(t) \rangle$, we first expand the gradient of the BCE loss giving
\[
    \nabla_\theta \ell(y,\theta)= \frac{1}{N}\sum_{i=1}^N(p_i-y_i)J_i,
\]
where $p_i=\sigma(f(x_i))$ and $J_i=\nabla_\theta f(x_i)$. Since $p$ is independent of $y$, subtracting the two gradients gives
\[
\nabla\ell_2(t)=\nabla\ell_1(t)-\frac{1}{N} \sum_{i=1}^N (y_{2,i}-y_{1,i})J_i.
\]
Substituting into the inner product, we find that
\begin{align}
-\langle\nabla_\theta\ell_1(t),\nabla_\theta\ell_2(t)\rangle
&= \left\langle\nabla\ell_1(t),\,\frac{1}{N}\sum_{i=1}^N(y_{2i}-y_{1i})J_i\right\rangle - \|\nabla \ell_1(t)\|_2^2\\
&\le \left\langle\nabla\ell_1(t),\,\frac{1}{N}\sum_{i=1}^N(y_{2i}-y_{1i})J_i\right\rangle\\
&\le \frac{d_H}{N}\|\nabla \ell_{1}(t)\|_{2} \max_{i\in\perp} \|J_i\|_2
\end{align}
where $\perp$ denotes the set where the labels disagree, i.e. where $y_{1,i}-y_{2,i} \ne 0$. The last line is derived by restricting the sum to the $d_H$ nonzero disagreement indices, applying Cauchy-Schwarz, and then a uniform Jacobian bound. 

Plugging in and summing gives $\varsigma \le r\eta T M + O(\eta^2T)$. Applying this bound to Equation~\ref{eq1} proves the theorem.
\end{proof}

\subsection{Proof of Theorem~\ref{thm:utop}}

To prove the lower bound, we begin with a lemma that provides a constraint for the optimization problem solved in the lower bound theorem that follows.

\begin{lemma}
\label{lem:constraint}
$a_1a_2 \leq c_2$ where $c_2=4^{-r}$.
\end{lemma}

\begin{proof}
We can expand the term-wise loss functions as
\[
\ell_{1,i}+\ell_{2,i}= 2\ln (1+e^{z})-(y_{1,i}+y_{2,i})z,
\]
where $z$ is the neural network output. Then, we have two cases. If $y_{1,i}=y_{2,i}$, then we get $\ell_{1,i}+\ell_{2,i}\ge0$. If not, then $y_{1,i}+y_{2,i}=1$ and we have
\[
2\ln(1+e^{z})- z = \ln(2+2\cosh z ) \ge \ln 4.
\]
Now, notice that
\begin{align}
(y_{1,i}-y_{2,i})^2 &= 
\begin{cases}
    0, & \text{if}\quad y_{1,i}=y_{2,i} \\
    1,  & \text{else}\quad
\end{cases} \\
&= \mathbf{1}_{\{y_{1,i}\ne y_{2,i}\}}.
\end{align}
Next, we get the total loss by summing over the data points, which gives
\begin{align}
 \ell_1+\ell_2 &= \frac{1}{N}\sum_{i=1}^N \left( \ell_{1,i}+\ell_{2,i} \right)\\
&\geq \frac{1}{N}\sum_{i=1}^N \mathbf{1}_{\{y_{1,i}\ne y_{2,i}\}} \ln4\\
&=\frac{\ln4}{N}d_H.
\end{align}

Finally, expanding $a_1a_2$ gives $a_1a_2 = e^{-(\ell_1+\ell_2)} \le e^{-\frac{\ln4}{N}d_H}=4^{-\frac{d_H}{N}}$.
\end{proof}

Using the lemma, we can now prove the lower bound on the utopia distance.

\begin{proof}
By Lemma~\ref{lem:constraint}, the constraint $a_1a_2\le c_2$ holds pointwise along the orbit. We therefore bound $\delta_U(t)$ pointwise by minimizing the squared instantaneous distance over admissible states $\mathbf{a}=(a_1,a_2)$:
\begin{equation*}
\arg\min_\mathbf{a} \quad \|\mathbf{a}-\mathbf{1}\|_2^2 \quad \text{s.t.} \quad a_1a_2\le c_2\tag{$\star$}.
\end{equation*}
First, notice that $\|\mathbf{a}-\mathbf{1}\|_2^2$ decreases as $a_1$ or $a_2$ increase. Thus, the solution to ($\star$) will lie on the hyperbola $a_1a_2=c_2$. We proceed by using Lagrange multipliers:
\begin{align}
    \mathcal{L}(\mathbf a,\lambda)&=\|\mathbf a-\mathbf 1\|_2^2 -\lambda(a_1a_2-c_2) =(1-a_1)^2+(1-a_2)^2-\lambda(a_1a_2-c_2)\\
    \nabla_\mathbf{a}\mathcal{L} &= -2\begin{bmatrix}1-a_1\\1-a_2\end{bmatrix} -\lambda \begin{bmatrix}a_2\\a_1\end{bmatrix}=\begin{bmatrix}0\\0\end{bmatrix}.
\end{align}
Solving, we find that $\lambda=\frac{-2}{a_2}(1-a_1)=\frac{-2}{a_1}(1-a_2)$. Cross multiplying and expanding gives
\[
(a_1-a_2)(1-(a_1+a_2))=0 \quad \implies \quad a_1=a_2 \quad \text{or} \quad a_1+a_2=1.
\]
The first case gives $a_1=a_2=\sqrt{c_2}$, resulting in $\|\mathbf{a}-\mathbf{1}\|_2 = \sqrt2 (1-\sqrt{c_2})=\sqrt2 (1- 2^{-r})$. In the second case, Vieta's formulas give that $a_1$ and $a_2$ are the roots of $x^2-x+c_2$. Hence, $a_1,a_2=\frac{1\pm \sqrt{1-4c_2}}{2}$. But we know that $a_1,a_2\in\mathbb{R}$, so $c_2\le \frac{1}{4}$ which gives $4^{-r} \le 4^{-1}$ which never happens since $0<r<1$. Thus, case 1 is the only possibility.

On the hyperbola $a_1a_2=c_2$, with $a_1,a_2\in[0,1]$, we parameterize $a_2=c_2/a_1$ and then minimize $f(a_1) = (1-a_1)^2 + \left(1-c_2/a_1\right)^2$. By symmetry, the minimum is attained at $a_1=a_2=\sqrt{c_2}$ since $f''(\sqrt{c_2})>0$.

Thus, $\delta_U(t) \geq \sqrt{2}(1-\sqrt{c_2})=\sqrt{2}(1-2^{-r})$ pointwise, and averaging over the cycle gives $d_U \geq \sqrt{2}(1-2^{-r})$.
\end{proof}

\subsection{EWC Corollary}

\begin{corollary}
\label{cor:EWC}
Suppose the Fisher information matrix is diagonal, the orbit remains sufficiently near some $\theta_1^*\in \Theta_1$ for $\ell_1$ to admit a quadratic approximation, $\eta$ is small enough that second-order update terms may be neglected, and all assumptions of Theorem~\ref{thm:rob} hold. Let $E = \frac{1}{T}\sum_{t=T}^{2T-1} \|\nabla \ell_1(t)\|_2^2$. If both losses use BCE and $rM \ge E(1+\lambda)$, then
\[
\Delta a_1 \le a_1(0)\bigl(e^{\eta T(rM - E)}-1\bigr);
\]
if the loss for environment 2 is replaced with the EWC loss $\ell_2^{\mathrm{EWC}}$, then
\[
\Delta a_1 \le a_1(0)\bigl(e^{\eta T(rM - E(1+\lambda))}-1\bigr).
\]
\end{corollary}

\begin{proof}
Recall that in proving Theorem~\ref{thm:rob}, we needed to upper bound the inner product of the two loss gradients. Replacing $\ell_2$ with $\ell_2^{\mathrm{EWC}}$ gives
\begin{align*}
-\langle \nabla \ell_1(\theta), \nabla \ell_2^{\mathrm{EWC}}(\theta) \rangle &= -\langle \nabla \ell_1(\theta), \nabla \ell_2(\theta) + \lambda F_1(\theta^*_{1}) (\theta - \theta^*_{1}) \rangle\\
&= -\langle \nabla \ell_1(\theta), \nabla \ell_2(\theta) \rangle - \lambda \langle \nabla \ell_1(\theta), F_1(\theta^*_{1}) (\theta - \theta^*_{1}) \rangle\\
&= \langle \nabla \ell_1(\theta), \frac{1}{N}\sum_{i=1}^N (y_{1,i}-y_{2,i}) J_i \rangle \\
&\qquad - \|\nabla \ell_1(\theta)\|_2^2 - \lambda \langle \nabla \ell_1(\theta), F_1(\theta^*_{1}) (\theta - \theta^*_{1}) \rangle.
\end{align*}
Because $\theta_1^*$ is a critical point of $\ell_1$, we have $\nabla \ell_1(\theta_1^*)=0$, so a first-order Taylor expansion of the \emph{gradient field} about $\theta_1^*$ gives
\[
\nabla \ell_1(\theta) = \nabla^2 \ell_1(\theta_1^*)(\theta - \theta_1^*) + O(\|\theta - \theta_1^*\|_2^2).
\]
With BCE loss the Hessian at the optimum is approximately the (diagonal) Fisher information matrix, $\nabla^2 \ell_1(\theta_1^*)\approx F_1(\theta_1^*)$, so $F_1(\theta^*_{1})(\theta - \theta^*_{1})\approx\nabla \ell_1(\theta)$ as full vectors, aligned in direction and approximately equal in magnitude, rather than merely comparable in norm. Hence $\langle \nabla \ell_1(\theta),\, F_1(\theta_1^*)(\theta - \theta_1^*)\rangle \approx \|\nabla \ell_1(\theta)\|_2^2 \ge 0$, so the EWC term strictly lowers the bound, and
\[
-\langle \nabla \ell_1(\theta), \nabla \ell_2^{\mathrm{EWC}}(\theta) \rangle \le rM - (1+\lambda) \|\nabla \ell_1(\theta)\|_2^2.
\]
If we let $E = \frac{1}{T}\sum_{t=T}^{2T-1} \|\nabla \ell_1(t)\|_2^2$, following the same steps as in Theorem~\ref{thm:rob} gives
\[
\Delta a_1 \le a_1(0)(e^{\eta T(rM - E(1+\lambda))}-1).
\]
The BCE result is recovered by setting $\lambda=0$. The condition $rM \ge E(1+\lambda)$ ensures that the exponent is nonnegative, so the bound is meaningful.
\end{proof}

\section{Choice of Accuracy Proxy}
\label{app:accuracy}

The utopia geometry requires a monotone map from loss to the unit square. Besides $a=e^{-\ell}$ we
compute three alternatives on every evaluation: a reciprocal transform $1/(1+\ell)$, direct $0/1$ accuracy, and
a Brier-style score $1-\overline{(p-y)^2}$. Across $7{,}212$ runs the four proxies rank-correlate strongly
(Table~\ref{tab:proxy}), and the optimal-reach trend of Section~\ref{sec:results} replicates for the exp and
reciprocal proxies (power-law $R^2\approx0.97$--$0.99$), confirming that the conclusions are not artifacts of
the specific $e^{-\ell}$ choice. The Theorem~\ref{thm:utop} floor is claimed only for $a=e^{-\ell}$.

\begin{table}[htbp]
\centering
\caption{Spearman rank correlations between the four performance proxies, computed per run
($n=7{,}212$; all $p<10^{-3}$).}
\label{tab:proxy}
\begin{tabular}{lcccc}
\toprule
& $e^{-\ell}$ & reciprocal & accuracy & Brier \\
\midrule
$e^{-\ell}$   & $1$      & $0.998$ & $0.823$ & $0.941$ \\
reciprocal    & $0.998$  & $1$     & $0.830$ & $0.943$ \\
accuracy      & $0.823$  & $0.830$ & $1$     & $0.939$ \\
Brier         & $0.941$  & $0.943$ & $0.939$ & $1$     \\
\bottomrule
\end{tabular}
\end{table}

\section{Theorem~\ref{thm:rob} Monotonicity Audit}
\label{app:monotonicity}

Theorem~\ref{thm:rob} assumes $a_1(t)$ decreases through the $y_2$ phase. We audited every steady-state
$y_2$ phase in the Boolean reach grid ($56{,}748$ phases). Using the endpoint-inclusive convention (the value just before the first $y_2$ update, followed by every post-update value), strict pointwise decrease held in $97.5\%$ of phases; the weaker endpoint condition---$a_1$ largest at the phase start and smallest at its end, which is what the bound actually needs---held in $98.8\%$; and the median normalized positive variation was
$0$. This variation---the total within-phase increase in $a_1$ divided by its total decrease, and zero for a strictly decreasing phase---is sharply concentrated at $0$ with a thin right tail: the non-monotone phases show only small, transient upticks in $a_1$ rather than sustained increases, which is why the endpoint condition holds even where strict pointwise monotonicity fails. Violations are mild and heterogeneous in $r$. The assumption is thus a close but not universal description of the dynamics.

\begin{figure}[htbp]
\centering
\includegraphics[width=0.62\textwidth]{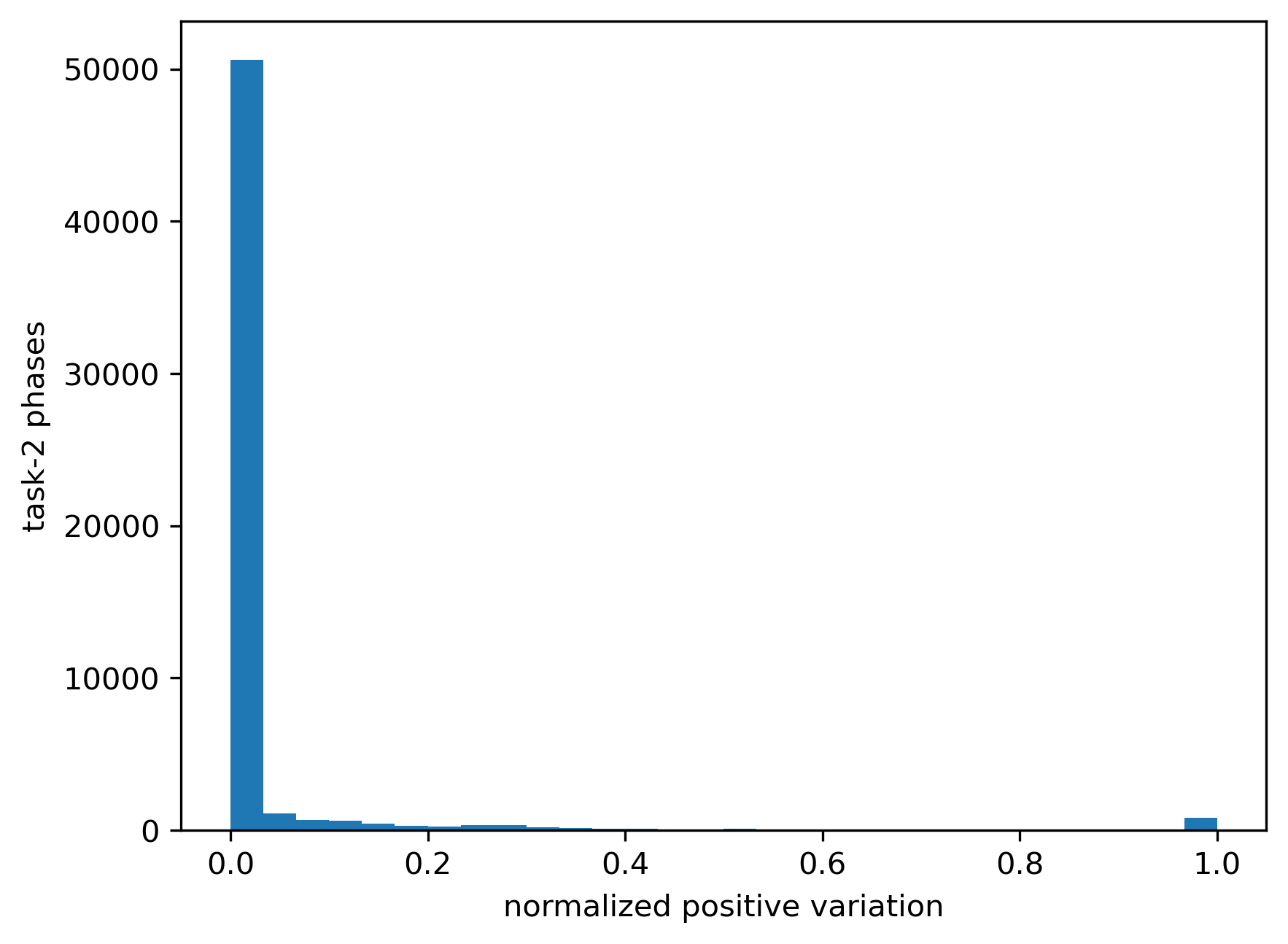}
\caption{Monotonicity audit of Theorem~\ref{thm:rob}. The $y_2$ phase is strictly decreasing in $97.5\%$ of phases, and the weaker endpoint condition holds in $98.8\%$. The median normalized positive variation is $0$, and the distribution is sharply concentrated at $0$ with a heterogeneous right tail.}
\label{fig:monotonicity}
\end{figure}

\section{Robustness Experiments}
\label{app:robust}

\subsection{Held-out generalization on continuous inputs}

To test performance on unseen inputs rather than the training truth table, we replace Boolean pairs with
linear-separator tasks on isotropic Gaussian inputs in $\mathbb{R}^{10}$. Two unit weight vectors $w_1,w_2$ are
placed at angle $\theta=\pi r$, giving labels $y_j=\mathbb{I}[X\!\cdot\! w_j>0]$ whose expected disagreement is
exactly $r$, a continuous analogue of $d_H/N$. We draw independent train ($4096$), validation ($2048$), and
test ($4096$) sets and evaluate on all splits. Over $454$ runs the train--test gap in $d_U$ is small (mean
$|{\Delta}|=0.007$, median $0.004$, max $0.064$), so the plasticity geometry is not a
training-set memorization artifact. This grid is small and some small-model pretraining did not fully converge,
so we treat it as a consistency check.

\subsection{Optimizer and architecture}

Table~\ref{tab:robust} sweeps the optimizer and the architecture while holding the reach sweep fixed.
The reach-governed picture persists under minibatch SGD, momentum, and Adam. Varying width and depth over a
$16\times$ range in parameter count ($3{,}641$ to $59{,}041$) leaves $d_U$ essentially unchanged, giving direct
empirical support for the network-size independence that our architecture-agnostic bounds assume but do not
prove. Figure~\ref{fig:archreach} shows the reach curves are preserved across architectures.

\begin{table}[htbp]
\centering
\caption{Mean steady-state $d_U$ (exp proxy) across optimizer and architecture variants, each with its
own reach sweep over $9$ task pairs and $5$ seeds.}
\label{tab:robust}
\begin{tabular}{llcc}
\toprule
\textbf{Study} & \textbf{Variant} & \textbf{Params} & \textbf{Mean $d_U$} \\
\midrule
Optimizer    & full-batch SGD    & $15{,}121$ & $0.565$ \\
             & minibatch-16      & $15{,}121$ & $0.539$ \\
             & minibatch-32      & $15{,}121$ & $0.548$ \\
             & momentum $0.9$    & $15{,}121$ & $0.487$ \\
             & Adam              & $15{,}121$ & $0.456$ \\
\midrule
Architecture & shallow ($3\times40$) & $3{,}641$  & $0.626$ \\
             & baseline ($10\times40$) & $15{,}121$ & $0.568$ \\
             & wide ($10\times80$) & $59{,}041$ & $0.581$ \\
\bottomrule
\end{tabular}
\end{table}

\begin{figure}[htbp]
\centering
\includegraphics[width=0.62\textwidth]{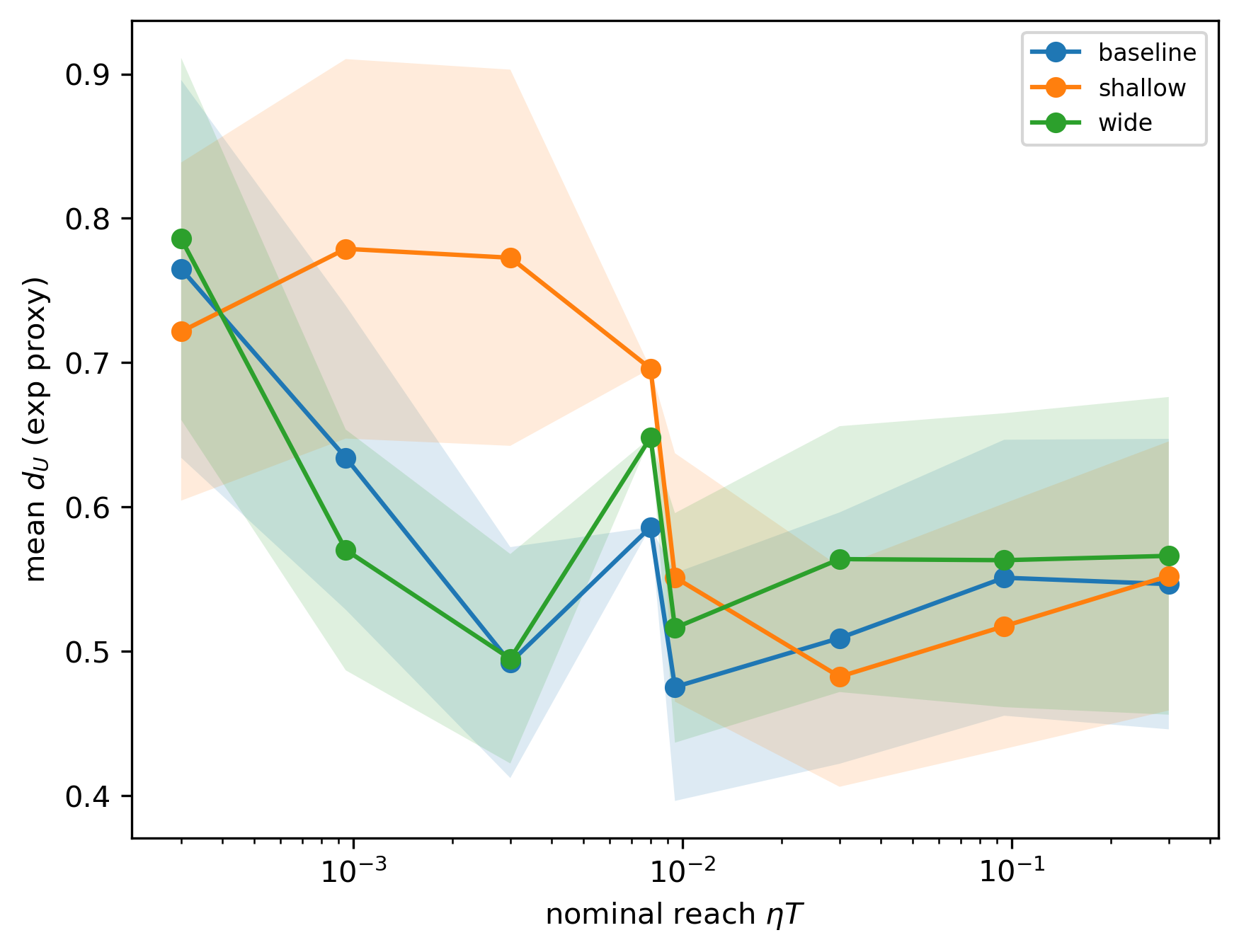}
\caption{$d_U$ versus nominal reach $\eta T$ for the three architectures; the optimum location and depth
are preserved across a $16\times$ parameter range.}
\label{fig:archreach}
\end{figure}

\end{document}